\documentclass{lms}

\usepackage{amscd}
\usepackage{amsfonts}
\usepackage{amsmath}
\usepackage{amssymb}
\usepackage{setspace}
\usepackage{version}

\newtheorem{theorem}{Theorem}[section]
\newtheorem{lemma}[theorem]{Lemma}

\newtheorem{corollary}[theorem]{Corollary}

\newtheorem{conjecture}[theorem]{Conjecture}

\renewcommand{\leq}{\leqslant}
\renewcommand{\geq}{\geqslant}
\def\eps{\varepsilon}

\def\Q{\mathbf{Q}}
\def\N{\mathbf{N}}
\def\E{\mathbf{E}}
\def\P{\mathbf{P}}

\numberwithin{equation}{section}

\begin{document}

\title{Commuting probabilities of finite groups}



\author{Sean Eberhard}


\maketitle
\def\cP{\mathcal{P}}
\begin{abstract}
The commuting probability of a finite group is defined to be the probability that two randomly chosen group elements commute. Let $\cP\subset(0,1]$ be the set of commuting probabilities of all finite groups. We prove that every point of $\cP$ is nearly an Egyptian fraction of bounded complexity. As a corollary we deduce two conjectures of Keith Joseph from 1977: all limit points of $\cP$ are rational, and $\cP$ is well ordered by $>$. We also prove analogous theorems for bilinear maps of abelian groups.
\end{abstract}

\tableofcontents

\section{Introduction}

Suppose we measure the abelianness of a finite group $G$ by counting the number of pairs of elements of $G$ which commute. Call
\[
	\Pr(G) = \P_{x,y\in G}(xy=yx) = \frac{1}{|G|^2} |\{(x,y)\in G^2 : xy=yx\}|
\]
the \emph{commuting probability} of $G$. Then a possibly surprising first observation is that a group $G$ with $\Pr(G)\approx 1$ must actually satisfy $\Pr(G)=1$. In fact if $\Pr(G)<1$ then $\Pr(G)\leq 5/8$. After such an observation it is natural to wonder what the rest of the set
\[
	\cP=\{\Pr(G): G\text{ a finite group}\}
\]
looks like. For instance, is there some $\eps>0$ such that if $\Pr(G)<5/8$ then $\Pr(G)\leq 5/8-\eps$? Is there some interval in which $\cP$ is dense? These sorts of questions were first studied in general by Keith Joseph~\cite{joseph1,joseph2}, who made the following three conjectures.

\begin{conjecture}[Joseph's conjectures]$ $\leavevmode
\begin{enumerate}
	\item[J1.] All limit points of $\cP$ are rational.
	\item[J2.] $\cP$ is well ordered by $>$.
	\item[J3.] $\{0\}\cup\cP$ is closed.
\end{enumerate}
\end{conjecture}

Note that conjectures J1 and J2, if true, answer our questions above about the structure of $\cP$, for J1 implies that $\cP$ is nowhere dense, and J2 implies that to every $p\in\cP$ we can associate some $\eps>0$ such that $(p-\eps,p)\cap\cP=\emptyset$. Progress on J1 and J2 has been slow, however. The best partial result to date is due to Hegarty~\cite{hegarty}, who proved that J1 and J2 hold for the set $\cP\cap(2/9,1]$.

From Hegarty's work one can begin to see a connection between commuting probability and so-called Egyptian fractions. The purpose of the present paper is to further develop this connection, and to use it to prove J1 and J2.

\def\cE{\mathcal{E}}
Define the \emph{Egyptian complexity} $\cE(q)$ of a rational number $q>0$ to be the least positive integer $m$ such that $q$ can be written as a sum of reciprocals
\[
	q = 1/n_1 + \cdots + 1/n_m,
\]
with each $n_i$ a positive integer, agreeing that $\cE(0)=0$ and that $\cE(x)=\infty$ if $x$ is irrational. We prove the following structure theorem for the values of $\Pr(G)$, which roughly asserts that commuting probabilities are nearly Egyptian fractions of bounded complexity.

\begin{theorem}\label{thm:main}
For every decreasing function $\eta:\N\to(0,1)$ there is some $M=M(\eta)\in\N$ such that every commuting probability $\Pr(G)$ has the form $q+\eps$, where $\cE(q)\leq M$ and $0\leq\eps\leq \eta(\cE(q))$.
\end{theorem}
\begin{corollary}\label{cor:joseph}
All limit points of $\cP$ are rational, and $\cP$ is well ordered by $>$.
\end{corollary}

We also prove a version of the above theorem for bilinear maps, partly as a model problem and partly for independent interest. Given finite abelian groups $A$, $B$, $C$ and a bilinear map $\phi:A\times B\to C$, let
\[
	\Pr(\phi) = \P_{a\in A,b\in B}(\phi(a,b)=0) = \frac{1}{|A||B|}|\{(a,b)\in A\times B: \phi(a,b)=0\}|.
\]
Let $\cP_{\text{b}}$ be the set of all $\Pr(\phi)$, where $\phi$ is such a bilinear map.

\begin{theorem}\label{thm:mainbilinear}
For every decreasing function $\eta:\N\to(0,1)$ there is some $M=M(\eta)\in\N$ such that every bilinear zero probability $\Pr(\phi)$ has the form $q+\eps$, where $\cE(q)\leq M$ and $0\leq\eps\leq \eta(\cE(q))$.
\end{theorem}
\begin{corollary}\label{cor:josephbilinear}
All limit points of $\cP_\textup{b}$ are rational, and $\cP_{\textup{b}}$ is well ordered by $>$.
\end{corollary}

The proofs of Theorems~\ref{thm:main} and~\ref{thm:mainbilinear} rely on a theorem of Neumann~\cite{pneumann} which states that if a group $G$ is statistically close to abelian in the sense that $\Pr(G)$ is bounded away from $0$ then $G$ is structurally close to abelian in the sense that $G$ has a large abelian section. We prove an amplified version of this theorem in Section~\ref{sec:neumann} and we use it to deduce Theorems~\ref{thm:main} and~\ref{thm:mainbilinear} in Section~\ref{sec:main}. We deduce Joseph's conjectures J1 and J2 in Section~\ref{sec:joseph}.
%

Assuming J2 holds, Joseph also asked for the order type of $(\cP,>)$. We consider this question in Section~\ref{sec:ordertype}. By examining the proof of Theorem~\ref{thm:main} we reduce the number of possibilities for the order type to two.

\begin{theorem}\label{thm:ordertype}
The order type of $(\cP,>)$ is either $\omega^\omega$ or $\omega^{\omega^2}$.
\end{theorem}

The same theorem holds for $\cP_\textup{b}$.

\section{Neumann's theorem amplified}\label{sec:neumann}

\begin{lemma}
Let $G$ be a finite group and $X$ a symmetric subset of $G$ containing the identity. Then $\langle X\rangle = X^{3r}$ provided $(r+1)|X|>|G|$.
\end{lemma}
\begin{proof}
Suppose $x_i\in X^{3i+1}\setminus X^{3i}$ for each $i=0,\dots,r$. Then for each $i$ we have
\[
	x_i X \subset X^{3i+2} \setminus X^{3i-1},
\]
so $x_0 X,\dots,x_r X$ are disjoint subsets of $G$ each of size $|X|$, so
\[
	(r+1)|X|\leq |G|.
\]
Thus if $(r+1)|X|>|G|$ we must have $X^{3i + 1} = X^{3i}$ for some $i\leq r$, so we must have $\langle X\rangle = X^{3i} = X^{3r}$.
\end{proof}

For $\phi:A\times B\to C$ a bilinear map and $A'\leq A$ and $B'\leq B$ subgroups, we denote by $\phi(A',B')$ the group generated by the values $\phi(a',b')$ with $a'\in A'$, $b'\in B'$.

\begin{theorem}[(Neumann's theorem for bilinear maps)]\label{thm:neumannbilinear}
Let $\eps>0$, and let $\phi:A\times B\to C$ be a bilinear map of finite abelian groups such that $\Pr(\phi)\geq\eps$. Then there are subgroups $A'\leq A$ and $B'\leq B$ such that $|A/A'|$, $|B/B'|$ and $|\phi(A',B')|$ are each $\eps$-bounded.
\end{theorem}
\begin{proof}
Let $X\subset A$ be the set of $x\in A$ such that $|\ker\phi(x,\cdot)|\geq (\eps/2)|B|$, and let $A'$ be the group generated by $X$. Then $|X|\geq (\eps/2)|A|$, so $A'$ has index at most $2/\eps$ in $A$, and by the lemma every $a\in A'$ is a sum of at most $6/\eps$ elements of $X$, so for every $a\in A'$ we have $|\ker\phi(a,\cdot)|\geq(\eps/2)^{6/\eps}|B|$. Similarly, there is a subgroup $B'$ of $B$ of index at most $2/\eps$ such that for every $b\in B'$ we have $|\ker\phi(\cdot,b)|\geq(\eps/2)^{6/\eps}|A|$. Then for every $a\in A'$ the subgroup $\ker\phi(a,\cdot)\cap B'$ has index at most $(2/\eps)^{6/\eps}$ in $B'$ and for every $b\in B'$ the subgroup $\ker\phi(\cdot,b)\cap A'$ has index at most $(2/\eps)^{6/\eps}$ in $A'$.

Now consider any value $c$ of $\phi$ on $A'\times B'$, say $c=\phi(a,b)$. If we replace $a$ by any element $a'$ of $a+(\ker\phi(\cdot,b)\cap A')$ and then $b$ by any element $b'$ of $b+(\ker\phi(a',\cdot)\cap B')$ then we still have $\phi(a',b')=c$, so
\[ |\{(a',b')\in A'\times B' : \phi(a',b')=c\}| \geq (\eps/2)^{12/\eps} |A'||B'|,\]
so $\phi$ takes at most $(2/\eps)^{12/\eps}$ different values on $A'\times B'$. But every element of $\phi(A',B')$ is a sum of distinct values of $\phi$ on $A'\times B'$, since if say
\[
	c = \sum_{i=1}^m \phi(a_i,b_i)
\]
with the term $\phi(a_j,b_j)$ appearing twice then we can reduce the total number of terms by replacing $\phi(a_j,b_j)+\phi(a_j,b_j)$ with $\phi(2a_j,b_j)$. Thus $|\phi(A',B')|\leq 2^{(2/\eps)^{12/\eps}}$.
\end{proof}

We need a stronger variant of the above theorem which asserts the existence of subgroups $A'$ and $B'$ such that (1) $\phi(A',B')$ is small and (2) $A'\times B'$ contains almost all pairs $(a,b)\in A\times B$ such that $\phi(a,b)\in \phi(A',B')$, in particular almost all pairs such that $\phi(a,b)=0$. The precise formulation is the following.

\begin{theorem}[(Neumann's theorem for bilinear maps, amplified)]\label{thm:neumannbilinear2}
For every decreasing function $\eta:\N\to(0,1)$ there is some $M=M(\eta)$ such that the following holds. For every bilinear map $\phi:A\times B\to C$ there are subgroups $A'\leq A$ and $B'\leq B$ such that
\begin{enumerate}
	\item $|\phi(A',B')|\leq M$,
	\item\label{bineu2} with at most $\eta(|\phi(A',B')|)|A||B|$ exceptions, every pair $(a,b)\in A\times B$ such that $\phi(a,b)\in \phi(A',B')$ is contained in $A'\times B'$.
\end{enumerate}
\end{theorem}

We have not stated a bound on $|A/A'|$ or $|B/B'|$, but such a bound is implicit if $\Pr(\phi)\geq\eps$, since then
\[
	\eps \leq\Pr(\phi)\leq \frac{1}{|A/A'||B/B'|} + \eta(|\phi(A',B')|).
\]
Thus by ensuring $\eta(1)\leq\eps/2$ one automatically has $|A/A'||B/B'|\leq 2\eps^{-1}$.

\begin{proof}
If $\Pr(\phi)\leq\eta(1)$ then we can just take $A'=B'=\{0\}$, so assume otherwise. Then we can apply Theorem~\ref{thm:neumannbilinear} with $\eps=\eta(1)$. Let $A_1\leq A$ and $B_1\leq B$ be the resulting subgroups, let $C_1=\phi(A_1,B_1)$, and suppose that more than $\eta(|C_1|)|A||B|$ pairs $(a,b)\in(A\times B)\setminus(A_1\times B_1)$ satisfy $\phi(a,b)\in C_1$. Then there is some $(a,b)\in(A\times B)\setminus(A_1\times B_1)$, say with $a\notin A_1$, such that at least $\eta(|C_1|)|A_1||B_1|$ pairs $(a',b')\in A_1\times B_1$ satisfy
\[
	\phi(a+a',b+b') \in C_1,
\]
or equivalently
\[
	\phi(a,b) + \phi(a,b') + \phi(a',b) \in C_1.
\]
Then in particular for at least $\eta(|C_1|)|B_1|$ elements $b'\in B_1$ we must have
\[
	\phi(a,b')\in C_1.
\]
But this implies
\[
	|(C_1 + \phi(a,B_1))/C_1|\leq\eta(|C_1|)^{-1},
\]
so if we put $A_2 = A_1+\langle a\rangle$, $B_2=B_1$, $C_2=\phi(A_2,B_2) = C_1 + \phi(a,B_1)$, then $|C_2|\leq\eta(|C_1|)^{-1} |C_1|$, $|B/B_2|\leq |B/B_1|$, and $|A/A_2|<|A/A_1|$. Now we can repeat the argument with $A_2, B_2, C_2$ in place of $A_1, B_1, C_1$, but since $|A/A_1||B/B_1|$ is an $\eta(1)$-bounded integer and $|A/A_2||B/B_2|<|A/A_1||B/B_1|$ this process must end after an $\eta(1)$-bounded number of steps, at which time we will have the conclusion of the theorem.
\end{proof}

We now turn our attention to the commutator map on groups, which behaves enough like a bilinear map for the above arguments to be emulated. In an arbitrary group $G$ we write $[x,y]$ for the commutator $x^{-1}y^{-1}xy$ of two elements $x,y\in G$. We also write $x^y$ for the conjugate $y^{-1}xy$, and we will use the relation $[x,y] = x^{-1}x^y$. For $H,K\leq G$ we write $[H,K]$ for the group generated by all commutators $[h,k]$ with $h\in H, k\in K$.

\begin{theorem}[(Neumann's theorem)]\label{thm:neumann}
Let $\eps>0$, and let $G$ be a finite group such that $\Pr(G)\geq\eps$. Then $G$ has a normal $2$-step nilpotent subgroup $H$ of $\eps$-bounded index such that $|[H,H]|$ is $\eps$-bounded.
\end{theorem}
\begin{proof}
Let $X\subset G$ be the set of all $x\in G$ such that $|C_G(x)|\geq (\eps/2)|G|$, where $C_G(x)$ is the centraliser of $x$ in $G$, and let $K$ be the group generated by $X$. Then $|K|\geq(\eps/2)|G|$, so $K$ has index at most $2/\eps$ in $G$, and by the lemma every $k\in K$ is the product of at most $6/\eps$ elements of $K$, so for every $k\in K$ we have $|C_G(k)|\geq(\eps/2)^{6/\eps}|G|$. Thus also $|C_K(k)|\geq(\eps/2)^{6/\eps}|K|$.

Now consider a commutator $c=[x,y]$ of two elements $x,y\in K$. If we replace $x$ by any element $x'$ of $C_K(y)x$ and then $y$ by any element $y'$ of $C_K(x')y$ then we still have $[x',y']=c$, so
\[ |\{(x',y')\in K^2 : [x',y']=c\}| \geq (\eps/2)^{12/\eps} |K|^2,\]
so there at most $(2/\eps)^{12/\eps}$ distinct commutators of elements of $K$. Now a classical theorem of Schur (see~\cite[10.1.4]{robinson}) implies that $|[K,K]|$ is $\eps$-bounded.

To finish let $H = C_{K}([K,K])$. Then $H$ has $\eps$-bounded index in $K$, hence $\eps$-bounded index in $G$, and since $[H,H]\subset[K,K]$ we see that $H$ is $2$-step nilpotent and $|[H,H]|$ is $\eps$-bounded.
\end{proof}

Now as in the case of bilinear maps we can prove a stronger variant which asserts the existence of a normal subgroup $H$ such that $[H,H]$ is small and such that $H\times H$ contains almost all $(x,y)\in G\times G$ such that $[x,y]\in[H,H]$, in particular almost all commuting pairs. We will need the following generalisation of Schur's theorem due to Baer (see~\cite[14.5.2]{robinson}).

\begin{lemma}\label{lem:baer}
If $M$ and $N$ are normal subgroups of a group $G$ then $|[M,N]|$ is bounded by a function of $|M/C_M(N)|$ and $|N/C_N(M)|$.
\end{lemma}

\begin{theorem}[(Neumann's theorem, amplified)]\label{thm:neumann2}
For every decreasing function $\eta:\N\to(0,1)$ there is some $M=M(\eta)$ such that the following holds. Every finite group $G$ has a normal subgroup $H$ such that
\begin{enumerate}
	\item $|[H,H]|\leq M$,
	\item with at most $\eta(|[H,H]|)|G|^2$ exceptions, every pair $(x,y)\in G^2$ such that $[x,y]\in[H,H]$ is contained in $H^2$.
\end{enumerate}
\end{theorem}
\begin{proof}
If $\Pr(G)\leq\eta(1)$ then we can just take $H=1$, so assume otherwise. Then we can apply Theorem~\ref{thm:neumann} with $\eps=\eta(1)$. Let $K_1\leq G$ be the resulting subgroup, let $L_1=K_1$, and suppose that more than $\eta(|[K_1,L_1]|)/2\cdot|G|^2$ pairs $(x,y)\in G^2\setminus(K_1\times L_1)$ satisfy $[x,y]\in [K_1,L_1]$. Then there must be some $(x,y)\in G^2\setminus(K_1\times L_1)$, say with $x\notin K_1$, such that at least $\eta(|[K_1,L_1]|)/2\cdot |K_1||L_1|$ pairs $(k,l)\in K_1\times L_1$ satisfy
\[
	[xk,yl]\in[K_1,L_1].
\]
By using the commutator expansion formula
\begin{equation}\label{commutatorexpansion}
 [ab,cd] = [a,d]^b [b,d] [a,c]^{bd} [b,c]^d
\end{equation}
and some further rearrangement, we can rewrite this as
\[
	[x,l^{-1}]^{-1} [x,y] [k^{-1},y]^{-1} \in [K_1,L_1].
\]
This implies that for some $l_0\in L_1$ there are at least $\eta(|[K_1,L_1]|)/2\cdot |L_1|$ elements $l\in L_1$ such that
\[
	[x,l^{-1}]^{-1} [x,l_0^{-1}] \in [K_1,L_1],
\]
so for these $l$ we have
\[
	[x,l_0^{-1} l] = ([x,l^{-1}]^{-1} [x,l_0^{-1}])^l \in [K_1,L_1].
\]

Thus the subgroup $N_0\leq L_1$ defined by
\[
	N_0 = \{l\in L_1 : [x,l]\in[K_1,L_1]\}
\]
has index at most $2\eta(|[K_1,L_1]|)^{-1}$ in $L_1$, thus index at most
\[
	2\eta(|[K_1,L_1]|)^{-1} |G/L_1|
\]
in $G$. If $N$ is the largest normal subgroup of $G$ contained in $N_0$ then it follows that
\[
	|G/N|\leq (2\eta(|[K_1,L_1]|)^{-1}|G/L_1|)!.
\]
But note that if $K_2$ is the normal subgroup of $G$ generated by $K_1$ and $x$ then in fact
\[
	N = \{l\in L_1: [K_2,l]\subset[K_1,L_1]\},
\]
so
\[
	N/[K_1,L_1] = C_{L_1/[K_1,L_1]}(K_2/[K_1,L_1]).
\]
Since trivially
\[
	K_1/[K_1,L_1] \leq C_{K_2/[K_1,L_1]}(L_1/[K_1,L_1]),
\]
Lemma~\ref{lem:baer} implies that the size of
\[
	[K_2/[K_1,L_1], L_1/[K_1,L_1]] = [K_2,L_1]/[K_1,L_1]
\]
is bounded by a function of $|L_1/N| \leq |G/N|$ and $|K_2/K_1|\leq |G/K_1|$, and thus the size of $[K_2,L_1]$ is bounded by a function of $\eta(|[K_1,L_1]|)$.

Now we can repeat the argument with $K_2$ and $L_2=L_1$ in place of $K_1$ and $L_1$, but since $|G/K_1||G/L_1|$ is an $\eta(1)$-bounded integer and $|G/K_2||G/L_2|<|G/K_1||G/L_1|$ this process must end after an $\eta(1)$-bounded number of steps, at which time we will have normal subgroups $K,L\leq G$ such that
\begin{enumerate}
	\item $|[K,L]|\leq M$,
	\item\label{point2etaKL} with at most $\eta(|[K,L]|)/2\cdot |G|^2$ exceptions, every pair $(x,y)\in G^2$ such that $[x,y]\in[K,L]$ is contained in $K\times L$.
\end{enumerate}
But \ref{point2etaKL} implies that with at most $\eta(|[K,L]|)|G|^2$ exceptions every pair $(x,y)\in G^2$ such that $[x,y]\in[K,L]$ is contained in both $K\times L$ and $L\times K$, and hence in $(K\cap L)^2$, so because
\[
	[K\cap L,K\cap L]\subset[K,L]
\]
the conclusion of the theorem is satisfied by $H=K\cap L$.
\end{proof}

We pause to mention that Theorem~\ref{thm:neumann2} admits a rather clean formulation in terms of ultrafinite groups. Given a sequence of finite groups $(G_n)$ and a nonprincipal ultrafilter $p\in\beta\N\setminus\N$, the set $\prod_{n\to p} G_n$ of all sequences $(g_n)\in \prod G_n$ defined up to $p$-almost-everywhere equality forms a group, which we refer to as an \emph{ultrafinite group}. The properties of $G = \prod_{n\to p} G_n$ tend to reflect the asymptotic properties of $(G_n)$. A subset $S$ of $G$ is called \emph{internal} if it is defined by subsets $S_n\subset G_n$ in the same way, namely if $(s_n)\in S$ if and only if $s_n\in S_n$ for $p$-almost-all $n$, in which case we write $S=\prod_{n\to p} S_n$. Internal subsets can be measured by assigning to $S = \prod_{n\to p} S_n$ the standard part of the ultralimit of $|S_n|/|G_n|$ as $n\to p$. The resulting premeasure extends to a countably additive measure, called \emph{Loeb measure}, on the $\sigma$-algebra generated by the internal subsets. In this language Theorem~\ref{thm:neumann2} can be stated as follows.

\begin{theorem}[(Neumann's theorem, amplified, ultrafinitary version)]\label{thm:neumann2ultra}
Every ultrafinite group $G$ has an internal normal subgroup $H$ such that $[H,H]$ is finite and such that almost every pair $(x,y)\in G^2$ such that $[x,y]\in [H,H]$ is contained in $H^2$.
\end{theorem}

Given a subgroup $H\leq G$, let us temporarily refer to pairs $(x,y)\in G^2\setminus H^2$ such that $[x,y]\in[H,H]$ as \emph{bad pairs}. Then the theorem states that every ultrafinite group $G$ has an internal normal subgroup $H$ with finite commutator subgroup and almost no bad pairs.

\begin{proof}[that Theorem~\ref{thm:neumann2} implies Theorem~\ref{thm:neumann2ultra}] 
Suppose $G$ were an ultrafinite group such that every internal normal subgroup $H$ with finitely many commutators has a positive measure set of bad pairs.

Note for every $M$ there is some $\eta(M)>0$ such that if $H\leq G$ is an internal normal subgroup with at most $M$ distinct commutators then the set of bad pairs for $H$ has measure at least $\eta(M)$. Indeed if not then for every $k$ there is an internal normal subgroup $\prod_{n\to p} H_{n,k}$ with at most $M$ distinct commutators and at most a measure $1/k$ set of bad pairs, so the internal normal subgroup $H=\prod_{n\to p} H_{n,n}$ has at most $M$ distinct commutators and almost no bad pairs, contradicting our hypothesis about $G$.

Applying Theorem~\ref{thm:neumann2} then to $G_n$ and $\eta/2$, we find normal subgroups $H_n\leq G_n$ with bounded-size commutator subgroups, say $|[H_n,H_n]|=M$ for $p$-almost-all $n$, such that $H_n$ has at most $(\eta(|[H_n,H_n]|)/2)|G_n|^2$ bad pairs. But then $H=\prod_{n\to p} H_n$ has at most $M$ commutators and at most a measure $\eta(M)/2$ set of bad pairs, a contradiction.

Thus for every ultrafinite group $G$ there is an internal normal subgroup $H$ with finitely many commutators and almost no bad pairs. By Schur's theorem $[H,H]$ is also finite.
\end{proof}

\begin{proof}[that Theorem~\ref{thm:neumann2ultra} implies Theorem~\ref{thm:neumann2}]
If Theorem~\ref{thm:neumann2} failed then we would have some decreasing function $\eta:\N\to(0,1)$ and for every $n$ some finite group $G_n$ such that every normal subgroup $H_n\leq G_n$ with $|[H_n,H_n]|\leq n$ has at least $\eta(|[H_n,H_n]|)|G_n|^2$ bad pairs. Let $G=\prod_{n\to p} G_n$. By Theorem~\ref{thm:neumann2ultra} there is an internal normal subgroup $H=\prod_{n\to p} H_n$ of $G$ with $[H,H]$ finite and almost no bad pairs. But then for $p$-almost-all $n$ the group $H_n$ has $|[H_n,H_n]|\leq|[H,H]|\leq n$ and fewer than $\eta(|[H,H]|)|G_n|^2$ bad pairs, a contradiction.
\end{proof}

\section{The main theorem}\label{sec:main}

For an abelian group $A$ we denote by $\widehat{A}$ the group of characters $\gamma:A\to S^1$. Recall the size relation $|\widehat{A}|=|A|$ and the orthogonality relations
\begin{align*}
	\E_{a\in A} \gamma(a) &= 1_{\gamma=1},\\
	\E_{\gamma \in \widehat{A}} \gamma(a) &= 1_{a=0}.
\end{align*}

\begin{lemma}\label{lem:bilinear}
Let $A$, $B$, $C$ be finite abelian groups and $\phi:A\times B\to C$ a bilinear map. Then $\cE(\Pr(\phi))\leq |C|$.
\end{lemma}
\begin{proof}
By orthogonality of characters we have
\begin{align*}
	\Pr(\phi)
	&= \E_{a\in A} \E_{b\in B} 1_{\phi(a,b)=0}\\
	&= \E_{a\in A} \E_{b\in B} \E_{\gamma\in\widehat{C}} \gamma(\phi(a,b))\\
	&= \E_{a\in A} \E_{\gamma\in\widehat{C}} 1_{\gamma(\phi(a,B)) = 1}\\
	&= \E_{\gamma\in\widehat{C}} \left(\frac{1}{|A|} |\{a\in A: \gamma(\phi(a,B)) = 1\}|\right).
\end{align*}
But for fixed $\gamma\in\widehat{C}$ the set $\{a\in A: \gamma(\phi(a,B))=1\}$ is a subgroup of $A$, so the above formula expresses $\Pr(\phi)$ as a sum of $|C|$ terms of the form $1/n$ with $n$ a positive integer.
\end{proof}

\begin{proof}[of Theorem~\ref{thm:mainbilinear}]
Fix $\eta:\N\to(0,1)$ and $\phi:A\times B\to C$. Applying Theorem~\ref{thm:neumannbilinear2}, we find some $M=M(\eta)$ and subgroups $A'\leq A$ and $B'\leq B$ such that $|\phi(A',B')|\leq M$ and such that no more than $\eta(|\phi(A',B')|)|A||B|$ pairs $(a,b)\in (A\times B)\setminus(A'\times B')$ satisfy $\phi(a,b)=0$. Thus
\[
	\Pr(\phi) = \frac{1}{|A/A'||B/B'|}\Pr(\phi_{A'\times B'})+\eps,
\]
where
\[
	\cE(\Pr(\phi_{A'\times B'}))\leq|\phi(A',B')|\leq M
\]
by the lemma, and
\begin{align*}
	\eps 
	&= \frac{|\{(a,b)\in (A\times B)\setminus(A'\times B') : \phi(a,b)=0\}|}{|A||B|}\\
	&\leq \eta(|\phi(A',B')|)\\
	&\leq \eta(\cE(\Pr(\phi_{A'\times B'}))).
\end{align*}
\end{proof}

The proof of Theorem~\ref{thm:main} is similar, but to prove a suitable analogue of Lemma~\ref{lem:bilinear} we need the following theorem of Hall~\cite{phall}.

\begin{lemma}\label{lem:hall}
In any group $G$ the index of the second centre
\[
	Z_2(G) = \{g\in G: [g,G]\subset Z(G)\}
\]
is bounded by a function of $|[G,G]|$.
\end{lemma}

\begin{lemma}
Let $G$ be a finite group. Then $\cE(\Pr(G))\leq|G/Z_2(G)|\cdot|[G,G]|$. In particular by Hall's theorem $\cE(\Pr(G))$ is bounded by a function of $|[G,G]|$.
\end{lemma}
\begin{proof}
Let $A$ be the abelian group $[G,G]\cap Z(G)$, and let $Z_2$ be the second centre of $G$. Then by the orthognality relations we have
\begin{align*}
	\Pr(G)
	&= \E_{x\in G} \E_{y\in G} 1_{[x,y]=1}\\
	&= \E_{x\in G} \E_{y\in G} \E_{z\in Z_2} 1_{[x,yz]=1}\\
	&= \E_{x\in G} \E_{y\in G} \E_{z\in Z_2} \E_{\gamma\in\widehat{A}} 1_{[x,yz]\in A} \gamma([x,yz])\\
	&= \E_{x\in G} \E_{y\in G} \E_{z\in Z_2} \E_{\gamma\in\widehat{A}} 1_{[x,y]\in A} \gamma([x,yz]),
\end{align*}
since, by \eqref{commutatorexpansion}, $[x,yz]=[x,z][x,y]^z\in A$ if and only if $[x,y]\in A$. Moreover, if $[x,y]\in A$ then $[x,yz]=[x,z][x,y]$, so by orthogonality again we have
\begin{align*}
	\Pr(G) 
	&= \E_{x\in G} \E_{y\in G} \E_{z\in Z_2} \E_{\gamma\in\widehat{A}} 1_{[x,y]\in A} \gamma([x,z])\gamma([x,y]),\\
	&= \E_{x\in G} \E_{y\in G} \E_{\gamma\in\widehat{A}} 1_{[x,y]\in A} \left(\E_{z\in Z_2} \gamma([x,z])\right) \gamma([x,y]),\\
	&= \E_{x\in G} \E_{y\in G} \E_{\gamma\in\widehat{A}} 1_{[x,y]\in A} 1_{\gamma([x,Z_2])=1} \gamma([x,y]).
\end{align*}
For fixed $y\in G$, $\gamma\in\widehat{A}$, let
\[
	G_{y,\gamma} = \{x\in G: [x,y]\in A, \gamma([x,Z_2])=1\}.
\]
Then, again by \eqref{commutatorexpansion}, $G_{y,\gamma}$ is a subgroup of $G$ and $x\mapsto[x,y]$ defines a homomorphism $G_{y,\gamma}\to A$, so
\[
	\Pr(G) = \E_{y\in G} \E_{\gamma\in\widehat{A}} \frac{1}{|G/G_{y,\gamma}|} 1_{\gamma([G_{y,\gamma},y])=1}.
\]
Finally, the integrand here depends on $y$ only through $yZ_2$, so we can replace the expectation over $y\in G$ by an expectation over $yZ_2\in G/Z_2$, so
\[
	\cE(\Pr(G))\leq |G/Z_2| \cdot|A| \leq |G/Z_2|\cdot |[G,G]|.
\]
\end{proof}

\begin{proof}[of Theorem~\ref{thm:main}]
Fix $\eta:\N\to(0,1)$ and $G$. By the lemma we can find another decreasing function $\eta':\N\to(0,1)$ such that
\[
	\eta'(|[G,G]|) \leq \eta(\cE(\Pr(G)))
\]
for all finite groups $G$. Applying Theorem~\ref{thm:neumann2} with $\eta'$, we find some $M=M(\eta)$ and a subgroup $H\leq G$ such that $|[H,H]|\leq M$ and such that no more than $\eta'(|[H,H]|)|G|^2$ pairs $(x,y)\in G^2\setminus H^2$ satisfy $[x,y]=1$. Thus
\[
	\Pr(G) = \frac{1}{|G/H|^2}\Pr(H)+\eps,
\]
where $\cE(\Pr(H))$ is bounded by a function of $|[H,H]|\leq M$ by the lemma, and
\begin{align*}
	\eps 
	&= \frac{|\{(x,y)\in G^2\setminus H^2 : [x,y]=1\}|}{|G|^2}\\
	&\leq\eta'(|[H,H]|)\\
	&\leq\eta(\cE(\Pr(H)))
\end{align*}
by the choice of $\eta'$.
\end{proof}

\section{Joseph's conjectures}\label{sec:joseph}

The following lemma is well known.

\begin{lemma}
For every $x>0$ and $m\in\N$ the supremum of the set of $q<x$ such that $\cE(q)\leq m$ is strictly less than $x$.
\end{lemma}
\begin{proof}
Suppose for contradiction that $n_{1i}$, ..., $n_{mi}$ are $m$ sequences of positive integers such that for all $i$
\[
	1/n_{1i} + \cdots + 1/n_{mi} < x
\]
and
\[
	1/n_{1i} + \cdots + 1/n_{mi} \to x.
\]
After rearranging and passing to a subsequence we may assume that $n_{1i} = n_1, \dots, n_{ki} = n_k$ are constants while $n_{k+1,i}, \dots, n_{mi} \to \infty$. But then
\[
	1/n_{1i} + \cdots + 1/n_{mi} \to 1/n_1 + \cdots + 1/n_k < x,
\]
a contradiction.
\end{proof}

\begin{proof}[of Corollary~\ref{cor:joseph}]
Let $x>0$ be a limit point of $\cP=\{\Pr(G): G\text{ a finite group}\}$. We will prove that $x$ is rational, and that if $p_n\to x$ then $p_n\geq x$ for all but finitely many $n$.

For $m\in\N$ let
\[
	Q(m,x) = \sup\{q<x : \cE(q) \leq m\}
\]
and define
\[
	\eta_x(m) = (x - Q(m,x))/2.
\]
By the lemma $\eta_x(m)>0$ for every $m$, so by Theorem~\ref{thm:main} there is some $M=M(\eta_x)$ such that every $p\in\cP$ has the form $q+\eps$, where $\cE(q)\leq M$ and $0\leq \eps\leq \eta_x(\cE(q))$.

Fix some such $p=q+\eps$ and suppose $q<x$. Then
\[
	\eps\leq \eta_x(\cE(q)) \leq (x-q)/2,
\]
so
\[
	p = q+\eps \leq (q+x)/2 \leq (Q(M,x)+x)/2 = x - \eta_x(M),
\]
so $p$ is bounded away from $x$. Thus if $p_n=q_n+\eps_n\to x$ then we must have $p_n\geq q_n\geq x$ for all but finitely many $n$. In particular $q_n\to x$, but the set of Egyptian fractions of complexity at most $M$ is closed, so this implies $\cE(x)\leq M$, so $x\in\Q$.
\end{proof}

Corollary \ref{cor:josephbilinear} is proved in exactly the same way.

\section{The order type of $\cP$}\label{sec:ordertype}

\def\ccP{\overline{\cP}}
Having shown in previous sections that $(\cP,>)$ is well ordered, we show in this final section that $(\cP,>)$ has order type either $\omega^\omega$ or $\omega^{\omega^2}$. First we need a standard definition.
%

For $X$ a closed subset of $[0,1]$ let $X'\subset X$ be the set of limit points of $X$. Iterating this operation, define $X^\alpha$ for ordinals $\alpha$ as follows: 
\begin{align*}
	&X^0 = X,\\
	&X^{\alpha+1} = (X^\alpha)',\\
	&X^\alpha = \bigcap_{\beta<\alpha} X^\beta \quad\text{ if }\alpha\text{ is a limit ordinal.}
\end{align*}
If $X$ is countable then there is a unique countable ordinal $\alpha$ for which $X^\alpha$ is finite and nonempty; we call $\alpha$ the \emph{Cantor-Bendixson rank} of $X$. If $X$ happens to be well ordered by $>$ then its order type is at most $\omega^\alpha + 1$, and if $X^\alpha=\{0\}$ and $\alpha>0$ then in fact the order type of $X$ is exactly $\omega^\alpha+1$. (For a detailed introduction to Cantor-Bendixson rank see Dasgupta~\cite[Chapter~16]{dasgupta}.)

\begin{lemma}
Let $X$ be a countably infinite closed subset of $[0,1]$ closed under multiplication, and let $\alpha$ be the Cantor-Bendixson rank of $X$. Then $X^\alpha=\{0\}$ and $\alpha=\omega^\beta$ for some ordinal $\beta$.
\end{lemma}
\begin{proof}
By induction on $\gamma$ if $x\in X^\gamma$ and $y\in X$ and $y>0$ then
\[
	xy\in X^\gamma.
\]
Hence by induction on $\delta$ if $x\in X^\gamma$ and $x>0$ and $y\in X^\delta$ and $y>0$ then
\[
	xy\in X^{\gamma+\delta}.
\]

Suppose $x\in X^\gamma$ and $x>0$. Fix $y\in X\cap(0,1)$. Then for all $n$ we have
\[
	xy^n\in X^\gamma,
\]
and $xy^n\to 0$, so
\[
	0\in X^{\gamma+1}.
\]
Hence we must have $X^\alpha=\{0\}$.

Now suppose $\gamma<\alpha$. Since $0\in X^\alpha\subset X^{\gamma+1}$ there must be some $x\in X^\gamma\cap(0,1)$. But then for all $n$ we have
\[
	x^n\in X^{\gamma\cdot n},
\]
so
\[
	0\in X^{\gamma \cdot\omega}.
\]
We deduce that
\[
	\alpha \geq \sup_{\gamma<\alpha} (\gamma\cdot\omega).
\]
Let $\omega^\beta$ be the largest power of $\omega$ such that $\omega^\beta\leq\alpha$. If $\omega^\beta<\alpha$ then
\[
	\alpha \geq \sup_{\gamma<\alpha} (\gamma\cdot\omega)\geq \omega^\beta\cdot\omega = \omega^{\beta+1},
\]
a contradiction, so we must have $\alpha=\omega^\beta$.
\end{proof}

Let $\ccP$ be the closure of $\cP$ in $[0,1]$. By the formula
\[
	\Pr(G\times H) = \Pr(G)\Pr(H)
\]
we know that $\cP$, and hence $\ccP$, is closed under multiplication, so if $\alpha$ is the Cantor-Bendixson rank of $\ccP$ then the lemma and the previous discussion implies that $\ccP$ has order type $\omega^\alpha+1$, so $\cP$ has order type $\omega^\alpha$, and moreover $\alpha=\omega^\beta$ for some $\beta$. Since for instance $1/2\in \ccP'$ we know that $\beta>0$. We will prove that $\alpha\leq\omega^2$, and thus $\alpha\in\{\omega,\omega^2\}$.

%

For $n\in\N$ let $\cE_n = \{q: \cE(q)\leq n\}$ be the set of Egyptian fractions of complexity at most $n$. The following lemma follows from the proofs of Theorem~\ref{thm:main}, Corollary~\ref{cor:joseph}, and Theorem~\ref{thm:neumann2}.

\begin{lemma}
For every $\eps_0>0$ there exist $k\in\N$ and a function $m:(0,1]\to\N$ such that for all $\eps_1,\dots,\eps_k>0$ the set
\[
	\ccP\cap[0,\eps_0]^c\cap\bigcap_{i=0}^{k-1}(\cE_{m(\eps_i)} + [0,\eps_{i+1}])^c
\]
is finite. 
\end{lemma}
\begin{proof}
Define $\eta_x(m) = (x - Q(m,x))/2$ as in the proof of Corollary~\ref{cor:joseph}. By inspecting the proofs of Theorem~\ref{thm:main} and Theorem~\ref{thm:neumann2}, we see that if $x>\eps_0$ the constant $M=M(\eta_x)$ can be taken to be the result of iterating some function
\[
	t\mapsto b(\eta_x(h(t)))
\]
some $n(\eps_0)$ times starting with $n(\eps_0)$, where
\begin{itemize}
	\item $b:(0,1]\to\N$ is a decreasing function coming from Baer's theorem (Lemma~\ref{lem:baer}),
	\item $h:\N\to\N$ is an increasing function coming from Hall's theorem (Lemma~\ref{lem:hall}),
	\item $n:(0,1]\to\N$ is a decreasing function coming from Neumann's theorem (Theorem~\ref{thm:neumann}).
\end{itemize}

Let $k=n(\eps_0)+1$ and $m(\eps)=\max(h(b(\eps)),h(n(\eps_0)))$, and suppose
\begin{equation}\label{xcp}
	x\in\cP\cap[0,\eps_0]^c\cap\bigcap_{i=0}^{k-1}(\cE_{m(\eps_i)} + [0,\eps_{i+1}])^c.
\end{equation}
Define the sequence $t_0,t_1,\dots,t_k$ by
\begin{align*}
	t_0 &= n(\eps_0),\\
	t_{i+1} &= b(\eta_x(h(t_i))) \quad\text{for }0\leq i<k.
\end{align*}
Then inductively
\begin{align*}
	h(t_i) &\leq m(\eps_i),\\
	\eta_x(h(t_i)) &\geq \eps_{i+1},\\
	t_{i+1} &\leq b(\eps_{i+1}),\\
	h(t_{i+1}) &\leq m(\eps_{i+1})
\end{align*}
for all $i$ in the range $0\leq i\leq k-1$, so
\[
	\eta_x(M) = \eta_x(t_{k-1}) \geq \eta_x(h(t_{k-1})) \geq \eps_k.
\]
But from the proof of Corollary~\ref{cor:joseph} we know that
\[
	(x-\eta_x(M),x)\cap\cP = \emptyset,
\]
so there can be at most $1/\eps_k$ elements $x$ in the set \eqref{xcp}, and the lemma follows.
\end{proof}

\begin{proof}[of Theorem~\ref{thm:ordertype}]
By the previous discussion it suffices to prove that $\alpha\leq\omega^2$. By the lemma we have
\[
	\ccP\subset[0,\eps_0]\cup\bigcup_{i=0}^{k-1}(\cE_{m(\eps_i)} + [0,\eps_{i+1}]) \cup F
\]
for some finite set $F$. From the rule $(X\cup Y)'=X'\cup Y'$ we have
\[
	\ccP'\subset[0,\eps_0]\cup\bigcup_{i=0}^{k-1}(\cE_{m(\eps_i)} + [0,\eps_{i+1}]).
\]
But since this holds for all $\eps_k>0$ we have
\[
	\ccP'\subset[0,\eps_0]\cup\bigcup_{i=0}^{k-2}(\cE_{m(\eps_i)} + [0,\eps_{i+1}])\cup\cE_{m(\eps_{k-1})}.
\]
Now using $\cE_n'=\cE_{n-1}$ and the rule $(X\cup Y)'=X'\cup Y'$ again we have
\[
	\ccP^{1+m(\eps_{k-1})}\subset[0,\eps_0]\cup\bigcup_{i=0}^{k-2}(\cE_{m(\eps_i)} + [0,\eps_{i+1}]).
\]
In particular
\[
	\ccP^{\omega}\subset[0,\eps_0]\cup\bigcup_{i=0}^{k-2}(\cE_{m(\eps_i)} + [0,\eps_{i+1}]).
\]
Repeating this argument another $k-1$ times, we have
\begin{align*}
	&\ccP^{\omega\cdot 2} \subset[0,\eps_0]\cup\bigcup_{i=0}^{k-3}(\cE_{m(\eps_i)} + [0,\eps_{i+1}]),\\
	&\quad\vdots\\
	&\ccP^{\omega\cdot(k-1)} \subset[0,\eps_0]\cup(\cE_{m(\eps_0)} + [0,\eps_1]),\\
	&\ccP^{\omega\cdot k} \subset [0,\eps_0].
\end{align*}
Thus
\[
	\ccP^{\omega^2}\subset[0,\eps_0]
\]
for all $\eps_0>0$, so
\[
	\ccP^{\omega^2}\subset\{0\},
\]
as claimed.
%
\end{proof}

The same argument applies unchanged in the case of $\cP_\text{b}$.

{\bf Acknowledgement.} I thank Freddie Manners for the idea in the proof of Theorem~\ref{thm:neumannbilinear} of bounding the size of $\phi(A',B')$ by showing that every value is taken many times, an idea which greatly simplified the proof.

\bibliography{nil}
\bibliographystyle{alpha}

\affiliationone{
   Sean Eberhard\\
   Mathematical Institute\\
	 Andrew Wiles Building\\
   Radcliffe Observatory Quarter\\
   Woodstock Road\\
   Oxford\\
   OX2 6GG\\
   UK\\
   \email{eberhard@maths.ox.ac.uk}
   }
\end{document}